\documentclass[letterpaper, 11pt]{article}
\usepackage{amsmath, amsthm, amssymb}
\usepackage[numbers,sort]{natbib}
\usepackage{bm}
\usepackage{graphicx}
\usepackage{fullpage}
\usepackage[colorlinks,linkcolor=blue,citecolor=blue,urlcolor=magenta,linktocpage,plainpages=false]{hyperref}\usepackage{xcolor}
\usepackage{xspace}
\usepackage{url}
\usepackage[noend]{algpseudocode}
\usepackage{libertine}
\usepackage{libertinust1math}
\usepackage[ruled,vlined]{algorithm2e}
\usepackage{caption}
\usepackage{setspace}
\usepackage[left=1in, right=1in, top=1in, bottom=1in]{geometry}
\usepackage{authblk}

\newtheorem{lemma}{Lemma}
\newtheorem{theorem}{Theorem}
\newtheorem{cor}{Corollary}

\newtheorem{remark}{Remark}
\newtheorem{claim}{Claim}
\newtheorem{definition}{Definition}

\newtheorem{example}{Example}

\newcommand{\R}[0]{\ensuremath \mathbb{R}}
\newcommand{\Z}[0]{\ensuremath \mathbb{Z}}

\newcommand{\ones}{\bm{1}}

\newcommand{\T}{\mathcal{T}}

\newcommand{\leaves}{\textrm{leaves}}

\newcommand{\mom}[1]{{\left\vert\kern-0.25ex\left\vert\kern-0.25ex\left\vert #1 \right\vert\kern-0.25ex\right\vert\kern-0.25ex\right\vert}}

\makeatletter
\newcommand{\algmargin}{\the\ALG@thistlm}
\makeatother
\newlength{\whilewidth}
\algnewcommand{\parState}[1]{\State%
  \parbox[t]{\dimexpr\linewidth-\algmargin}{\strut #1\strut}}
\algtext*{Until}%

\newcounter{mynotes}
\newcommand{\conv}{\textrm{conv}}
\newcommand{\proj}{\textrm{proj}}

\title{Branch-and-Bound versus Lift-and-Project Relaxations \\ in Combinatorial Optimization}
\author[1]{Gérard Cornuéjols\thanks{gc0v@andrew.cmu.edu} }
\author[2]{Yatharth Dubey\thanks{ydubey@illinois.edu}}
\affil[1]{Tepper School of Business, Carnegie Mellon University}
\affil[2]{Industrial and Enterprise Systems Engineering, University of Illinois at Urbana-Champaign}
\date{\today }

\begin{document}

\maketitle

\begin{abstract}
 In this paper, we consider a theoretical framework for comparing branch-and-bound with classical lift-and-project hierarchies. We simplify our analysis of streamlining the definition of branch-and-bound. We introduce ``skewed $k$-trees'' which give a hierarchy of relaxations that is incomparable to that of Sherali-Adams, and we show that it is much better for some instances. We also give an example where lift-and-project does very well and branch-and-bound does not. Finally, we study the set of branch-and-bound trees of height at most $k$ and effectively ``squeeze'' their effectiveness between two well-known lift-and-project procedures. 
\end{abstract}

\section{Introduction}

In integer programming, branching and cutting are two basic algorithmic strategies at the heart of current solvers. For any application or subproblem at hand, deciding whether to branch or to cut can drastically impact the computing time. Understanding this trade-off has been of interest for a long time and is addressed, for example, in papers such as \cite{basu2023complexity} and the references therein. In this paper, we consider a new theoretical framework for comparing these two strategies. 

Many combinatorial optimization problems can be written in the following way: $\max \{cx : x \in P \cap \{0,1\}^n\}$, where $c \in \R^n$, $P := \{x \in [0,1]^n : Ax \geq b\}$, with $A \in \R^{m \times n}, b \in \R^m$. Define $P_I := \conv(P \cap \{0,1\}^n)$. A fundamental goal in integer programming and combinatorial optimization is to obtain a relaxation $Q$ such that $P_I \subseteq Q \subseteq P$, where $Q$ is as ``close'' to $P_I$ as possible; see \cite{Conforti_Cornuéjols_Zambelli_2014, schrijver1998theory, schrijver2003combinatorial, wolsey1999integer} for more on integer programming and combinatorial optimization. 

One way to obtain such relaxations is via an \textit{extended formulation}, i.e. a polyhedron $Q_E := \{(x,y) : A'x + B'y \geq b'\} \subseteq \R^{n + p}$, where $A' \in \R^{r \times n}, B' \in \R^{r \times p}, b' \in \R^{r}$, such that $P_I \subseteq \proj_x \left(Q_E\right) \subseteq P$. When $r,p$ are polynomial in $n$, we refer to $Q_E$ as a \textit{compact extended formulation}; see \cite{conforti2010extended} for more on extended formulations.  Hierarchies of extended formulations can be obtained by iterating the procedures of 
Balas-Ceria-Cornuejols (BCC) \cite{balas1993lift}, and Lovasz-Schrijver (LS) \cite{lovasz1991cones}, or by applying the Sherali-Adams procedure at various levels (SA) \cite{sherali1990hierarchy}; see \cite{laurent2003comparison} for a comparison between the latter two procedures and \cite{cornuejols2001elementary} for a broader comparison of cut-operators. 

Another way of generating extended formulations arises from the branch-and-bound method (BB) \cite{land2010automatic}. Consider a partial enumeration tree $\T$ where the leaves represent subproblems $P_v := P \cap \{x : C_v\}$ where $C_v$ denotes the set of branching constraints from the root to node $v$. The set $\bigcup_{v \in \leaves(\T)} P_v$ is contained in $P$ and contains all feasible integer solutions. In other words, letting $\T(P) := \conv\left( \bigcup_{v \in \leaves(\T)} P_v \right)$, we have $P_I \subseteq \T(P) \subseteq P$. The polytope $\T(P)$ has an extended formulation by a theorem of Balas \cite{balas1985disjunctive} and this formulation is polynomial size if the number of leaves of $\T$ is polynomial. 


The potential of extended formulations arising from BB is made evident by the following example. 

\begin{example}
Let $K_n$ be the complete graph on $n$ vertices $V$ and let $P = \{x \in \R^n_+ : x_u + x_v \leq 1 \ \forall \ u,v \in V\}$ be the corresponding fractional stable set polytope. Observe that the inequality $\sum_{v \in V} x_v \leq 1$ is valid for $P_I$. 

BB is able to obtain an extended formulation $Q_{BB}$, with $O(n^2)$ variables and $O(n^3)$ constraints, such that $\sum_{v \in V} x_v \leq 1$ is valid for $Q_{BB}$. This is because there is a BB tree $\T$ of size $O(n)$ such that $\T(P) = P_I$. In particular, observe that after branching on some variable $x_v$, the side fixing $x_v = 1$ only contains the point where all other $x_u = 0$ for $u \not = v$; therefore, the desired inequality is valid for any tree where each variable is branched on once and only once, and a node is branched on only if it is either the root or a child corresponding to the branch $x_v = 0$ for some $v \in V$ (i.e. a node with a constraint $x_v = 1$ for any $v \in V$ is never branched on); furthermore, all leaves of such a tree are integral. 

By contrast, SA produces an extended formulation $Q_{SA}$ with exponentially many variables and constraints for $\sum_{v \in V} x_v \leq 1$ to be valid for $Q_{SA}$. In particular, the inequality $\sum_{v \in V} x_v \leq 1$ is not valid for $SA^{n-3}(P)$, which has an exponential size extended formulation. This is because the point $\frac{1}{k+2}\ones \in SA^k (P)$, for any $1 \leq k \leq n-2$ (see Section 6.1 of \cite{laurent2003comparison}). 
\end{example}
This example illustrates the fact that extended formulations derived from BB can outperform those generated by the SA hierarchy.
The goal of this work is to better understand how BB compares to lift-and-project on other families of discrete optimization problems. 
Specifically, we seek to identify when one approach (BB or lift-and-project) provides a tight compact extended formulation while the other does not. 

We note that \cite{singh2010improving} studied the Gomory-Chvàtal (GC) operator \cite{gomory2010outline, chvatal1973edmonds} in the same spirit: they showed that GC significantly outperforms SA for the maximum matching problem in $k$-uniform hypergraphs; however, GC performs as poorly as SA for max cut, unique label cover, and $k\text{-}{CSP}_q$. They concluded that the ``positive result gives strong motivation for studying GC cuts as an algorithmic technique.'' This paper shows that a similar conclusion can be made for studying BB trees.

In Section 2, we introduce branch-and-bound formally, in the form used in this paper. The polytope  $\T(P)$ is derived from the corresponding enumeration tree. We also discuss the different variants of lift-and-project that we will use as a comparison, and we prove technical results that will be used later in the paper.

In Section 3, we introduce a special form of enumeration tree called ``skewed $k$-tree''. We show that, for fixed $k$, these trees have polynomial size. We identify two combinatorial optinization problems where skewed $k$-tree solve the problem to optimality. Namely, we show that, for the maximum clique problem on sparse graphs, and the uniform knapsack problem with small capacity, we can construct a polynomial size skewed $k$-tree such that  $\T(P) = P_I$. By contrast, for these two combinatorial optimization problems,  any polynomial-size formulation in the Sherali-Adams hierarchy has a projection $SA^k(P)$ into the $x$-space  that is strictly contains $P_I$. 

In Section 4, we prove a similar result for combinatorial problems defined by no-good constraints, using a different family of enumeration trees. We construct a polynomial size BB tree $\T$ such that  $\T(P) = P_I$ but even when there is just one no-good point, any SA formulation that produces $P_I$ has an exponential number of variables and constraints.

In Section 5, we give an example where lift-and project does very well while banch-and-bound does poorly. Specifically, we exhibit a polytope $P$ such that no branch-and-bound tree of size less than $2^{(n-1)/6}$ produces $P_I$, while lift-and-project generates $P_I$ in two rounds. 

In Section 6, we consider the set of BB trees with height at most $k$. We show that this family of trees produces a bound at least as strong as any sequential convexification of $k$ variables as per the BCC procedure, but no stronger than $k$ iterations of the canonical lift-and-project procedure. In other words, we show that the bound obtained from this  family of  BB trees is squeezed between two classical  lift-and-project bounds.

\section{Extended Formulations} \label{section2}

We begin by formally defining branch-and-bound in the form we study here; this comes almost directly from the interpretation of branch-and-bound in Section 2 of \cite{dey2023lower}, but here we only consider trees constructed via variable disjunctions (as opposed to general disjunctions). We then introduce the lift-and-project operators of interest in this paper.

\subsection{Relaxations Based on Branch-and-Bound}

We simplify our analysis of branch-and-bound by removing two conditions typically assumed of branch-and-bound: (i) the requirement that the partitioning into two subproblems (which correspond to the child nodes of the given node) is done in  such a way that the optimal LP solution of the parent node is not included in either subproblem, and (ii) branching is not done on pruned nodes. By removing these conditions, we can talk about a BB tree independent of the underlying polytope -- it is just a binary tree (i.e. each node has $0$ or $2$ child nodes). The root-node has an empty set of \emph{branching constraints}. If a node has two child nodes, these are obtained by applying a disjunction $x_j = 0 \,\vee\, x_j = 1$ for some $j \in [n]$, where each child node adds one of these constraints to its set of branching constraints together with all the branching constraints of the parent node. Finally, note that since a BB tree is a binary tree, the total number of nodes of a BB tree with $N$ leaf-nodes is $2N - 1$.

\begin{definition}\label{def:abstract_bb}
Given a branch-and-bound tree $\mathcal{T}$, applied to a polytope $P \subseteq [0,1]^n$, and a node $v$ of the tree:
\begin{itemize}\vspace{-8pt}
\item We denote the number of nodes of the branch-and-bound tree $\mathcal{T}$ by $|\mathcal{T}|$. This is what is termed \emph{the size} of this tree.
\item We denote by $C_v$ the set of branching constraints of $v$ (as explained above, these are the constraints added by the branch-and-bound tree along the path from the root node to $v$).
\item For any leaf $v \in \leaves(\T)$, we denote $J_v^0 := \{j : \{x_j = 0\} \in C_v\}$ and $J_v^1 := \{j : \{x_j = 1\} \in C_v\}$.
\item We say that the \emph{height} of a node $v \in \T$ is the number of its branching constraints; formally, $h(v) = |C_v|$ for any $v \in \T$. The height of $\T$ is $h(\T) := \max_{v \in \T} h(v)$. Observe  that $h(\text{root}) = 0$, and thus $h(\T) \leq n$. 
\item We call the feasible region defined by the LP relaxation $P$ and the branching constraints at node $v$ the \emph{atom} of this node; formally, we denote $P_v = P \cap \{x : C_v\}$ is the atom corresponding to $v$.
\item We let $\mathcal{T}(P)$ denote the {convex hull of the union of the atoms corresponding to the} leaves of $\mathcal{T}$ when applied to polytope $P$, i.e., $$\mathcal{T}(P) = \conv \left(\bigcup_{v \in \leaves(\mathcal{T})} P_v\right).$$
\item For $x^* \in P \setminus P_I$, we say that $\T$ \emph{separates} $x^*$ if $x^* \not \in \T(P)$. 
\end{itemize} 
\end{definition}

\begin{theorem}[Balas \cite{balas1985disjunctive}, see Theorem 5.1 of \cite{conforti2010extended}] \label{thm:bb_ef_size}
There is an extended formulation for $\T(P)$ with $O(|\T|(n+1))$ variables and $O(|\T|m)$ constraints. 
\end{theorem}

We will study different types of BB trees in Sections~\ref{section2}-\ref{section6}. In Section~\ref{section3} we introduce skewed $k$-trees and we show that they can be advantageous for generating the relaxation $\mathcal{T}(P)$. In Section~\ref{section4} we exhibit a BB tree that works well for combinatorial problems defined by no-good constraints. In Section~\ref{section5} we analyze BB trees with bounded number of leaves and in Section~\ref{section6} BB trees with bounded height.

\subsection{Lift-and-Project} 

In this section we review three classical lift-and-project operators and we present technical lemmas that will be useful in subsequent sections.

First we introduce the canonical lift-and-project operator; see Section 5.4 of \cite{Conforti_Cornuéjols_Zambelli_2014}. 

$$L(P) = \bigcap_{i \in [n]} \conv((P \cap \{x : x_i = 0\}) \cup (P \cap \{x : x_i = 1\}))$$

Let    $L^0(P) := P$ and     $L^k(P) := L\left(L^{k-1}(P)\right)$ for any $k \geq 1$.

We will use the following technical results. 

\begin{lemma}\label{lem:nest_l}
$L^k(P) \cap \{x: x_j = a\} = L^k(\{x \in P : x_j = a\})$ for any $j \in [n], k \in [n], a \in \{0,1\}$. 
\end{lemma}
\begin{proof}
First, we prove the lemma for $k = 1$. Also, we assume $a = 0$ as the argument for $a = 1$ follows very easily. 

First, we show the easier containment, $L(P \cap \{x : x_j = 0\}) \subseteq L(P) \cap \{x : x_j = 0\}$. Clearly $L(P \cap \{x : x_j = 0\}) \subseteq \{x : x_j = 0\}$ and it is also easy to see that $L(P \cap \{x : x_j = 0\}) \subseteq L(P)$ by monotonicity.

Now we show $L(P) \cap \{x : x_j = 0\} \subseteq L(P \cap \{x : x_j = 0\})$. Let $\bar{x} \in L(P) \cap \{x : x_j = 0\}$. 
It suffices to prove the claim that $\bar{x} \in \conv((P \cap \{x : x_j=0, x_i = 0\}) \cup (P \cap \{x : x_j=0, x_i = 1\}))$ for each $i$. The claim is true for $i = j$ since $\bar{x} \in L(P) \cap \{x : x_j = 0\}$ implies $\bar{x} \in P \cap \{x : x_j = 0\}$. Now consider $i \not= j$. The claim is true when $\bar{x}_i =0$ or $\bar{x}_i =1$ because $\bar{x} \in L(P)$ implies  $\bar{x} \in P$ and therefore  
$\bar{x} \in P \cap \{x : x_j=0, x_i = 0\}$ or $\bar{x} \in P \cap \{x : x_j=0, x_i = 1\}$. So we may assume $0 < \bar{x}_i <1$.
Since $x \in L(P)$, there is some $x^0 \in \{x \in P : x_i = 0\}$ and $x^1 \in \{x \in P : x_i = 1\}$ such that $x \in \conv(x^0, x^1)$. Moreover, it must hold that $x^0_j = x^1_j = 0$, since otherwise would imply $x_j > 0$. This proves the claim.

We now prove the lemma inductively. Assume $L^k(P) \cap \{x: x_j = a\} = L^k(\{x \in P : x_j = a\})$. Then $L^{k+1}(P) \cap \{x : x_j = a\} = L(L^k(P)) \cap \{x : x_j = a\} = L(L^k(P) \cap \{x : x_j = a\}) = L(L^k (P \cap \{x : x_j = a\})) = L^{k+1}(P \cap \{x : x_j = a\})$, where the first and last equality follow from the definition of $L^k$, the second from the base case of the induction proven above, and the third from the inductive hypothesis. 
\end{proof}

\begin{lemma}[Lemma 3.1 of \cite{dash2015relative}]\label{lem:height_lemma}
Consider $(a,b) \in \R^n \times \R$ with $a \not = 0$ and let $s^1, ..., s^n$ be affinely independent points in $\{x \in \R^n : ax = b\}$. Consider $b' > b$ and let $R$ be a bounded and non-empty subset of $\{x \in \R^n : ax \geq b'\}$. Then, there exists a point $x \in \bigcap_{r \in R} \conv(s^1,...,s^n,r)$ satisfying the strict inequality $ax > b$. 
\end{lemma}

In the statement and proof of the next theorem, we use the following notation. For any $J \subseteq [n]$, let $P^J := P \cap \{x : x_j = 0 \ \forall \ j \in J\}$. 

\begin{theorem}\label{thm:lp_lower_bound}
Let $P \subseteq [0,1]^n$ be a polytope. Assume $P_I$ is full-dimensional, and let $cx \geq \delta$ be a facet of $P_I$. Let $ t \in [n-1]$. For all $J \subset [n]$ with $|J| = n-t$, assume that $P^J$ satisfies the following properties.
\begin{enumerate}
    \item $\dim\left(\left(P^J\right)_I\right) = n - |J| = t$
    \item $\{x \in P^J : cx \geq \delta \}$ is a facet of $\left(P^J\right)_I$
    \item there is some $x^J \in P^J$ such that $cx^J < \delta$.
\end{enumerate}
Then, $L^{n-t}(P) \not = P_I$. Indeed, there is some $\bar{x} \in L^{n-t}(P)$ with $c\bar{x} < \delta$. 
\end{theorem}

\begin{proof}
We will use the following claim for the proof. 

\begin{claim}\label{claim:smaller_J}
Consider any $J \subset [n]$ with $|J| \leq n-t$. Properties 1 and 2 of the theorem hold for $P^J$. 
\end{claim}
\begin{proof}[Proof of Claim \ref{claim:smaller_J}]
We proceed by induction: Let $k \geq t$ and suppose Properties 1 and 2 hold for all $J \subset [n]$ with $|J| = n-k$. Now consider any $J \subset [n]$ with $|J| = n-(k + 1)$. First, we will show that $\{x \in P^J : cx = \delta \}$ is a facet of $\left(P^J\right)_I$. In particular, we will demonstrate $k+1$ affinely independent points $S_J \subset P \cap \{0,1\}^n$ such that for any $x \in S_J$, it holds that $cx = \delta$ and $x_j = 0$ for all $j \in J$. 

Fix any $i \not \in J$ and let $J' = J \cup \{i\}$. Observe that, by the induction hypothesis, there is a set $S_{J'} \subset \{0,1\}^n$ of $k$ affinely independent points such that for any $x \in S_{J'}$, it holds that $cx = \delta$ and $x_j = 0$ for all $j \in J'$; furthermore, observe that $S_{J'} \subset P^{J'} \subset P^J$. Therefore, it suffices to show that there is a point in $\{0,1\}^n$ with $x_i = 1$ with $cx = \delta$ and $x_j = 0$ for all $j \in J$, since such a point would be affinely independent from the set $S_{J'}$, since $x_i = 0$ for any $x \in S_{J'}$.

Consider any $i' \not \in J'$ (such $i'$ exists since $k \geq 1$ implies $|J'| \leq n-1$) and let $J'' = J \cup \{i'\}$. Let $S_{J''}$ be defined similarly to $S_{J'}$ and observe that there is a point $x'' \in S_{J''}$ with $x''_i = 1$, since  $\dim\left(P^{J''}\right)_I=k$ and we already have $k$ equations $x_j =0$ for $j \in J''$. Clearly, since $x'' \in S_{J''}$ it holds that $cx'' = \delta$ and $x''_j = 0$ for all $j \in J''$. Therefore, $S_{J'} \cup \{x''\}$ is a set of $k+1$ affinely independent points in $P^J \cap \{0,1\}^n$ with $cx = \delta$, showing that $P^J$ satisfies Property 2 of the theorem. The proof showing $P^J$ satisfies Property 1 is exactly the same, but with $S_{J'}$ including an additional point in $\left(P^{J'}\right)_I$ with $cx > \delta$ (clearly this point is affinely independent from the points in the facet). 
\end{proof}

We now prove the theorem by induction. Assume for all $J \subset [n]$ such that $|J| = n - k$, there is a point $x^J \in L^{k-t}(P^J)$ with $cx^J < \delta$. The theorem assumes this is true for $k = 0$. Consider any $J \subseteq [n]$ such that $|J| = n - (k + 1)$, we will show that there is a point $x^J \in L^{(k + 1) - t}(P^J)$ with $cx^J < \delta$. It follows from the definition of $L(\cdot)$ and Lemma \ref{lem:nest_l} that,
$$L^{k-t+1}\left(P^J\right) = \bigcap_{i \in [n] \setminus J} \conv \left( L^{k-t}\left(P^{J \cup \{i\}}\right) \cup L^{k-t}(P \cap \{x : x_i = 1\}) \right).$$
By the induction hypothesis, the $i$-th term of the intersection contains a point $x^{J \cup \{i\}}$ such that $cx^{J \cup \{i\}} < \delta$ and by Claim \ref{claim:smaller_J} it contains a set $S_J$ of $k+1$ affinely independent points, each satisfying $cx = \delta$. Therefore, by Lemma \ref{lem:height_lemma}, there is a point $x^J \in \bigcap_{i \in [n] \setminus J} \conv\left(S_J \cup \left\{x^{J \cup \{i\}}\right\}\right)$ such that $cx^J < \delta$ as desired. 
\end{proof}

The second lift-and-project operator that we consider in this paper is based on sequential convexification, introduced in Section 2 of \cite{balas1993lift}. 

Let $k \in [n]$, and let $i_1, ..., i_k \in [n]$. Theorem 2.2 of \cite{balas1993lift} gives a method to find a compact extended formulation for the polytope $P_{i_1,...,i_k} = \conv(\{x \in P : x_j \in \{0,1\} \ \forall \ j \in \{i_1,...,i_k\}\})$. In this study, we will consider the following natural operator
$$B^k(P) = \bigcap_{\{i_1,...,i_k\} \subset [n]} P_{i_1,...,i_k}$$
Note that $L(P) = B^1(P)$ and $L^k(P) \subseteq B^k(P)$ for any $k \geq 1$ (this also follows from Theorem \ref{thm:bounded_height_main}).

Third, we consider the Sherali-Adams hierarchy; see Section 10.4.1 of \cite{Conforti_Cornuéjols_Zambelli_2014}. We denote by  $SA^k(P)$ the $k$th level polytope in the Sherali-Adams hierarchy. We have $P_I = SA^n(P) \subseteq  SA^{n-1}(P) \subseteq \ldots \subseteq SA(P) \subseteq P$.


\begin{remark} \label{rem:sa_ef_size}
The extended formulation for $SA^k(P)$ has $O(\binom{n}{k})$ variables and $O(\binom{n}{k})$ constraints. 
\end{remark}

\begin{remark}
It is well-known and easy to see that $SA^k(P) \subseteq L^k(P)$ for any $k \geq 1$ and polytope $P \subseteq [0,1]^n$; see Lemma 10.8 and Theorem 10.11 of \cite{Conforti_Cornuéjols_Zambelli_2014}. 
\end{remark}



\section{Skewed $k$-trees}\label{section3}

In this section, we present two problems that branch-and-bound is able to solve in polynomial time, while Sherali-Adams requires an exponential size extended formulation. This capability of branch-and-bound can be understood by considering a specific family of trees.

\begin{definition}
Let $k \in \Z_{\geq 1}$ and $\pi$ be a permutation of $[n]$. We construct a \emph{skewed k-tree} $\T^{\pi,k}$ as follows.

Let $\T_0$ consist of only the root, i.e. a single node $v$ with $C_v = \emptyset$. Then, for $i = \pi(1),...,\pi(n)$, do the following. For any leaf $v \in \leaves(\T_{i-1})$, if $|J^1_v| < k$: branch on variable $i$ at $v$. In other words, $\leaves(\T_i) = \left(\leaves(\T_{i-1}) \setminus \{v : |J^1_v| < k\}\right) \cup \left\{v^0, v^1 : C_{v^a} = C_v \cup \{x_i = a\}, a \in \{0,1\}\right\}$, and set $\T^{\pi,k} = \T_n$. 
\end{definition}

\begin{theorem}\label{lem:bb_k_subsets}
For any permutation $\pi$, the skewed k-tree $\T^{\pi,k}$ has size $n^{O(k)}$ and satisfies the following properties
\begin{enumerate}
    \item for each subset $J \subset [n]$ of $k' \leq k$ variables, there is exactly one leaf $v \in \leaves(\T^{\pi,k})$ such that $J_v^1 = J$ and
    \item for any leaf $v \in \leaves(\T^{\pi,k})$, either $|J_v^1| = k$ or $J_v^0 \cup J_v^1 = [n]$.
\end{enumerate}
\end{theorem}

\begin{proof}
We begin by showing $|\T^{\pi,k}| = n^{O(k)}$. Let $v^1,v^2 \in \leaves(\T)$. Observe that $J^1_{v^1} \not = J^1_{v^2}$, in particular $|J^1_{v^1} \triangle J^1_{v^2}| \geq 1$ (where $A \triangle B$ is the symmetric difference of the sets $A,B$): let $w$ be the most recent common ancestor (that of maximum height) of $v^1,v^2$, and let $\ell$ be the variable branched on at $w$ (so $w$ has height $\ell - 1$); then, clearly $\ell \in J^1_{v^1} \triangle J^1_{v^2}$. Furthermore, by construction, $|J_v^1| \leq k$ for any $v \in \leaves(\T^{\pi,k})$. Therefore, for each $v \in \leaves(\T^{\pi,k})$, there is a distinct size $\leq k$ subset of $[n]$, and there are only $\sum_{i \in [k]} \binom{n}{i} = n^{O(k)}$ such subsets, proving the desired bound on $|\T^{\pi,k}|$. 

First, we will show that $\T^{\pi,k}$ satisfies property 2 of the lemma. Since, $|J_v^1| \leq k$ for any $v \in \leaves(\T^{\pi,k})$, we only need to show that, for any $v \in \leaves(\T^{\pi,k})$ such that $|J_v^1| < k$, it must be that $J_v^0 \cup J_v^1 = [n]$. Let $v$ be such a leaf and, for sake of contradiction, let $j' = \min\{j \not \in J_v^0 \cup J_v^1\}$ (note that $j' \leq h(v) + 1$, since $|[n] \setminus (J_v^0 \cup J_v^1)| = n - h(v)$) and let $w$ be the ancestor of $v$ at height $j'-1$ (note that it is possible $w = v$, in the case $j' = h(v) + 1$). Then, $w \in \leaves(\T_{j'-1})$ and $|J^1_w| \leq |J^1_v| < k$, therefore, $j'$ is branched on at $w$ and therefore $j' \in J_v^0 \cup J_v^1$, achieving the desired contradiction. 

Finally, we will show that $\T^{\pi,k}$ satisfies property 1 of the lemma. Let $J \subset [n]$ of size $k' \leq k$. We will show that there exists a leaf $v \in \leaves(\T^{\pi,k})$ such that $J^1_v = J$, which will prove the desired result, since we have already argued that $J^1_{v^1} \not = J^1_{v^2}$ for distinct leaves $v^1,v^2$. Let $j' = \max\{j \in J\}$ and let $\bar{v}$ be a node be such that $C_{\bar{v}} = \left(\bigcup_{j \in [j'] \setminus J} \{x_j = 0\} \right) \cup \left( \bigcup_{j \in J} \{x_j = 1\} \right)$. Clearly by construction $\bar{v} \in T_{j'}$ since its parent $\bar{\bar{v}}$ fixes exactly variables $1,...,j'-1$ and has $|J_{\bar{\bar{v}}}^1| < k$. If $k' = k$, then $\bar{v} \in \leaves(\T^{\pi,k})$ and $J_v^1 = J$. If $k' < k$, then there is a leaf $v \in \leaves(\T^{\pi,k})$ with $C_v = C_{\bar{v}} \cup \left( \bigcup_{j' + 1 \leq j \leq n} \{x_j = 0\} \right)$. We can see that $v \in \T^{\pi,k}$ since any node $\bar{\bar{v}}$ that is an ancestor of $v$ will have $|J_{\bar{\bar{v}}}^1| < k$ and therefore will be branched on. 
    \end{proof}

\begin{remark}
Let $\mathcal{F} = \{\{x : x_{i_1} = ... = x_{i_k} = 1\}$ : $\{i_1,...,i_k\} \subset [n]\}$ be the set of all faces of $[0,1]^n$ defined by fixing $k$ variables to $1$. Suppose for any $F \in \mathcal{F}$, it holds that $P \cap F = \emptyset$ or $P \cap F = (P \cap F)_I$. Let $\T$ be \textbf{any} branch-and-bound tree that branches on a node if and only if the node is not empty nor integral, through a similar argument to that of Theorem \ref{lem:bb_k_subsets}, it holds that $\T(P) = P_I$ and $|\T| = n^{O(k)}$. 
\end{remark}

\begin{remark}
For any permutation $\pi$, the sequence of skewed $k$-trees $\{\mathcal{T}^{\pi,k} : k = 1,...,n\}$ corresponds to a hierarchy of relaxations. Let $P \subseteq [0,1]^n$, then observe $\T^{\pi, k'}(P) \subset \T^{\pi, k}(P)$ for any $k < k'$ and $\T^{\pi, n}(P) = P_I$. 
\end{remark}

At first glance, this seems similar to the $k$-th level of Sherali-Adams, which aims to take advantage of ``easy'' subsets of $k$ variables. However, in the following subsections, we will see classes of instances for which the above property allows branch-and-bound to succeed, while Sherali-Adams is unable to do so. Intuitively, this is because the $k$-th level of Sherali-Adams may not able to take advantage of inequalities that are valid for $P \cap \{x : C_v\}$ when $v \in \leaves(\T)$ has height $> k$, where as in the case of $\T^{\pi,k}$, it is ensured that either $k$ variables are fixed to one or all variables are fixed. 



\subsection{Uniform Knapsack with Small Capacity}

In this section, we apply Theorem \ref{lem:bb_k_subsets} to the problem of a uniform knapsack with small capacity.


\begin{theorem}\label{thm:knapsack}
Let $P \subset [0,1]^n$ be such that $\sum_{i \in [n]} x_i \leq k - \epsilon$ is valid for $P$, for any $\epsilon \in (0,1)$. Let $\T^{\pi,k}$ be a skewed k-tree, then $\T^{\pi,k}(P) = P_I$. Furthermore, there exists a $P$ such that $\sum_{i \in [n]} x_i \leq 2 - \frac{2}{q}$ is valid for $P$, yet $SA^{\lfloor n(q-2)/(q-1) \rfloor}(P) \not = P_I$, where $q$ can be any integer at least $3$. 
\end{theorem}

\begin{proof}
We begin by proving the positive result on branch-and-bound. In particular, we show that for every leaf $v \in \leaves(\T^{\pi,k})$, either its atom is empty (i.e. $P \cap \{x : C_v\} = \emptyset$) or integral (i.e. $P \cap \{x : C_v\}$ has only integral extreme points). Let $v \in \leaves(\T^{\pi,k})$ be such that $|J_v^1| = k$. Then, for any $x \in \{x : C_v\}$, it holds that $\sum_{i \in [n]} x_j = \sum_{j \in J_v^1} x_j = k > k - \epsilon$, which implies $x \not \in P$, and therefore $P \cap \{x : C_v\} = \emptyset$. Clearly, for any $v \in \leaves(\T^{\pi,k})$ such that $J_v^0 \cup J_v^1 = [n]$, it holds that $P \cap \{x : C_v\}$ is either empty or a single point in $\{0,1\}^n$. This concludes the proof of $\T^{\pi,k}(P) = P_I$. 

Now, we prove the lower bound on Sherali-Adams. Let $q$ be any integer at least $3$ and let $P = \{x \in [0,1]^n : \sum_{j \in [n]}qx_j \leq 2(q-1)\}$. Note that $\sum_{i \in [n]} x_i \leq 2 - \frac{2}{q}$ is clearly valid for $P$, and in particular we can write $P = \{x \in [0,1]^n : \sum_{i \in [n]}x_i \leq 2(\frac{q-1}{q})\}$. Theorem 5 of \cite{karlin2011integrality} shows that the point $\ones \cdot \frac{2(q-1)/q}{n + (t-1)(q-1)/q} \in SA^t(P)$. Therefore, when $t < n\frac{q-2}{q-1} + \frac{q}{q-1}$, it holds that $\sum_{i \in [n]} x_i \leq 1$ is not valid for $SA^t(P)$, proving the desired result. 
\end{proof}


\subsection{Max Clique on Sparse Graphs}\label{sec:clique}

One measure of graph density that has shown to be relevant in practical applications is the notion of \textit{degeneracy}.

\begin{definition}[\cite{lick1970k}]
A graph $G$ is $d$-degenerate if each of its subgraphs $G' \subseteq G$ has minimum degree at most $d$. The \textit{degeneracy} of $G$ is the smallest $d$ such that $G$ is $d$-degenerate. Moreover, it is a simple fact that the size $k$ of the maximum clique in $G$ satisfies $k \leq d+1$. 
\end{definition}

Observe that the degeneracy of a forest is $1$, of a cycle is $2$, of a planar graph is at most $5$, and of a $K_n$ is $n-1$; this demonstrates how degeneracy can be a reasonable measure of density. The parameterization by degeneracy is inspired by \cite{walteros2020maximum,naderi2022worst}. In particular, Table 2 of \cite{walteros2020maximum} shows that, for many real-world graphs, the degeneracy is several orders of magnitude smaller than the number of nodes, indicating that further study of algorithms on such graphs is worthwhile. We show that branch-and-bound is more effective for finding cliques on such graphs than Sherali-Adams. 

\begin{theorem}\label{thm:clique}
Let $G = (V,E)$ be a graph on $n$ vertices with maximum clique size $k$ and degeneracy $d$. Let $P = \{x \in \R^n_+ : x_u + x_v \leq 1 \ \forall \ uv \not \in E\}$ be the fractional clique polytope for $G$. Let $\T^{\pi,k}$ be a skewed k-tree, then $\T^{\pi,k}(P) = P_I$, while $SA^{\lfloor n/(d + 1) - 3 \rfloor}(P) \not = P_I$. 
\end{theorem}

\begin{proof}
We begin by showing $\T^{\pi,k}(P) = P_I$. In particular, we will show that for any leaf $z \in \leaves(\T^{\pi,k})$, its atom is either empty (i.e. $P \cap \{x : C_z\} = \emptyset$) or integral (i.e. $P \cap \{x : C_z\}$ has only integral extreme points). Consider any $z \in \leaves(\T^{\pi,k})$ with $|J_z^1| = k$. Suppose $J_z^1$ is not a clique. Then, there is some $u,v \in J_z^1$ such that $uv \not \in E$. However, $\{x_u = 1\}, \{x_v = 1\} \in C_z$, and therefore $P \cap \{x : C_z\} = \emptyset$ since $\{x : C_z\}$ violates $x_u + x_v \leq 1$. Now suppose $J_z^1$ is a clique, then it must be a maximum (and therefore maximal) clique. Consider any $u \not \in J_z^1$, there is some $v \in J_z^1$ such that $uv \not \in E$, and therefore, since $x_u + x_v \leq 1$ for any $x \in P$, it holds that $x_u = 0$ for any $x \in P \cap \{x : C_z\}$. This implies that $P \cap \{x : C_z\}$ is integral (containing only the point $x$ such that $x_j = 1$ for $j \in J_z^1$ and $x_j = 0$ otherwise). We have shown that any leaf $z \in \leaves(\T^{\pi,k})$ with $|J_z^1| = k$ is either empty or integral, and clearly any leaf $z \in \leaves(\T^{\pi,k})$ with $J_v^0 \cup J_v^1 = [n]$ is either empty or integral, implying the desired result. 

For the proof of the lower bound for Sherali-Adams, we will need the following observation of \cite{lick1970k}: a graph with degeneracy $d$ admits a vertex ordering $(v_1, v_2, ..., v_n)$ in which each vertex $v_i$ has at most $d$ neighbors to its right (i.e. $|N(v_i) \cap \{v_{i+1},...,v_n\}| \leq d$, where $N(v)$ is the open-neighborhood of $v$). Let $(v_1, ..., v_n)$ be such an ordering, and let $N^r(v_i) = N(v_i) \cap \{v_{i+1},...,v_n\}$ be the right-neighborhood of $v_i$.

We will now show that $SA^{\frac{n}{d + 1} - 3}(P) \not = P_I$. In particular, observe that $G$ has an stable set of size at least $\frac{n}{d + 1}$. We can construct such an stable set, call it $S$, greedily: start with $S = \{v_1\}$, then choose the minimum index variable $v_j$ such that $v_j \not \in \bigcup_{v \in S} N^r(v)$ and add it so that $S = S \cup \{v_j\}$. Observe that after choosing $k$ variables, at most $kd + k$ variables are ruled out (i.e. $|S \cup \bigcup_{v \in S} N^r(v)| \leq k + kd$). Therefore, for $k < \frac{n}{d+1}$, there is still a variable that can be added to $S$. Now, let $S$ be a maximum stable set $G$ of size at least $\frac{n}{d+1}$. Section 6.1 of \cite{laurent2003comparison} shows that $ \frac{1}{t+2}\ones \in SA^t(P)$, and therefore $\frac{1}{n/(d+1) - 1} \ones \in SA^{\frac{n}{d + 1} - 3}(P)$. Therefore, the inequality $\sum_{v \in S} x_v \leq 1$ is not valid for $SA^{\frac{n}{d + 1} - 3}(P)$, while it is clearly valid for $P_I$. 
\end{proof}


\section{Polynomially Many No-Good Constraints}\label{section4}

In this section, we give another family of instances where the branch-and-bound tree is of polynomial size, but the Sherali-Adams relaxation is not. 

\begin{theorem}\label{thm:no_good}
Let $P$ be defined by a set of no-good constraints: let $S \subseteq \{0,1\}^n$ and $P = \{x \in [0,1]^n : \sum_{j : s_j = 0}x_j + \sum_{j : s_j = 1} (1 - x_j) \geq \frac{1}{2} \ \forall \ s \in S \}$. There is a branch-and-bound tree $\T$ of size at most $3n|S|$ such that $\T(P) = P_I$. Furthermore, $SA^{n-1}(P) \not = P_I$ when $|S| = 1$ and $L^{n-1}(P) \not = P_I$ if there exists $z \in S$ such that $\{z' : \|z' - z\|_1 = 1\} \cap S = \emptyset$. 
\end{theorem}

\begin{proof}
We will describe a branch-and-bound tree $\T$ as in the theorem. Let each $s \in S$ define a leaf of $\T$: for each $s \in S$, there is a $v^s \in \leaves(\T)$ such that $\{x : x_j = s_j\} \in C_{v^s}$ for all $j \in [n]$ (i.e. $P \cap C_{v^s} = \{s\}$). Clearly there exists such a tree, since a complete binary tree with $2^n$ leaves is such an example. Let $u \in \T$ be a node that is not an ancestor of any of the leaves $\{v^s : s \in S\}$. We claim that the atom of $u$ is integral (i.e. $P \cap C_u$ has integer extreme points), and so these are leaves of $\T$. Therefore, each internal node of $\T$ is an ancestor of some leaf $v^s$, and since there are $|S|$ such leaves, there are at most $n|S|$ internal nodes, completing this portion of the proof.

We now prove the claim that the atom of $u$ is integral for any $u$ that is not an ancestor of some $v^s$. Observe that $\{x \in [0,1]^n : C_u\}$ defines some face of $[0,1]^n$ that is disjoint from $S$, i.e. $\{x \in [0,1]^n : C_u\} \cap S = \emptyset$. Then, any extreme point $z$ of $\{x \in [0,1]^n : C_u\}$ is such that $z \in \{0,1\}^n \setminus S$. By the definition of $P$, it follows $\{0,1\}^n \setminus S \subseteq P$, and therefore $z \in P \cap \{0,1\}^n$, concluding the proof of the desired claim. 

Finally, consider the case $S = \{0\}$. Proposition 19 of \cite{laurent2003comparison} shows that the point $\ones \cdot \frac{1}{2n-t} \in SA^t(P)$, implying $\sum_{i \in [n]} x_i \geq 1$ is not valid for $SA^{n-1}(P)$, which in turn shows $SA^{n-1}(P) \not = P_I$ as desired. Furthermore, suppose there is some $z \in S$ such that $\{z' : \|z' - z\|_1 = 1\} \cap S = \emptyset$, by symmetry we can assume $z = 0$. Then, $\sum_{j \in n}x_j \geq 1$ is a facet of $P_I$, and the corresponding properties of Theorem \ref{thm:lp_lower_bound} are satisfied with $t = 1$: let $J$ be any subset of $n-1$ variables, then $x_j \geq 1$ is a facet of $\left(P^J \right)_I$ where $j = [n] \setminus J$ and $\frac{1}{2}e^j \in P^J$ (vector with all zeros, except a half in the $j$-th component) is a point violating this facet. Then, Theorem \ref{thm:lp_lower_bound} implies the final statement of the theorem.
\end{proof}

\section{Limits of Branch-and-Bound}\label{section5}

In this section, we rule out general statements on the relative strength of branch-and-bound. For example, perhaps we would like to say something like the following. If lift-and-project does really well (e.g. $L^2(P) = P_I$), there is some \textit{small} branch-and-bound tree that separates each $x \in P \setminus P_I$. Below, we show that statements of this form are not true.

\begin{theorem} \label{thm:bb_vs_lp}
Let $n$ be any nonnegative integer such that $n-1$ is divisible by $6$. Let $\mathbb{T}^\ell_k$ be the set of branch-and-bound trees with at most $k$ leaves. There is a polytope $P \subseteq [0,1]^n$ such that 
$$P_I = L^2(P) \subsetneq \bigcap_{\T \in \mathbb{T}^\ell_{2^{(n-1)/6}}} \T (P).$$

In other words, there is a point $x \in P \setminus P_I$ such that no branch-and-bound tree of size $\leq 2^{(n-1)/6}$ separates $x$, while lift-and-project retrieves the integer hull in two rounds. 
\end{theorem}
\begin{proof}
We begin by defining the polytope of interest, $P$. Let $P \cap \{x \in \R^n : x_n = 0\} = \{x \in \{0,1\}^n : x_n = 0\}$. Further, let $P \cap \{x \in \R^n : x_n = 1\} = Q$, where $Q$ is the fractional stable set polytope defined below, intersected with the halfspace $\{x \in \R^{n-1} : \sum_{i \in [n-1]} x_i \geq \frac{n-1}{3} + \frac{1}{2} \}$. Let $m = \frac{n-1}{3}$, we define $G = (V,E)$ to be the disjoint union of $m$ triangles, in particular, $V = \bigcup_{i \in [m]}\{ia, ib, ic\}$ and $E = \bigcup_{i \in [m]} \{\{ia,ib\}, \{ib,ic\}, \{ia,ic\}\}$; so we define $Q = \{x \in \R^{n-1}_+ : x_u + x_v \leq 1 \ \forall \ \{u,v\} \in E, \sum_{v \in V} x_v \geq m + \frac{1}{2}\}$, and it should be clear that $Q \cap \{0,1\}^{n-1} = \emptyset$. 

First, we show $L^2(P) = P_I$. By definition and Lemma \ref{lem:nest_l}, $L^2(P) \subseteq \conv((L(P) \cap \{x : x_n = 0\}) \cup (L(P) \cap \{x : x_n = 1\})) = \conv(L(\{x \in P : x_n = 0\}) \cup L(\{x \in P : x_n = 1\}))$. We will show $L(\{x \in P : x_n = 1\}) = \emptyset$, which along with the fact that $\{x \in P : x_n = 0\} = P_I$, will prove the claim. To see that $L(\{x \in P : x_n = 1\}) = \emptyset$, observe that $x_{ia} + x_{ib} + x_{ic} \leq 1$ is valid for both $Q \cap \{x : x_{ia} = 0\}$ and $Q \cap \{x : x_{ia} = 1\}$, therefore, each of the $m$ such inequalities is valid for $L(Q)$, implying that $\sum_{v \in V} x_v \leq m$ is valid for $L(Q)$. This, along with the fact that $\sum_{v \in V} x_v \geq m + \frac{1}{2}$ is clearly valid for $L(Q)$ proves the desired claim.

Now we will show that there is some $x \in P \setminus P_I$ such that $x \in \T(P)$ for any $\T \in \mathbb{T}^\ell_{2^{(n-1)/6}}$ (note that $\frac{n-1}{6} = \frac{m}{2}$). Observe that if there is a leaf $z$ of $\T$ such that $\{x : x_n = 0\}$ is not in $C_z$ and there is some $i \in [m]$ such that no variable in $\{ia, ib, ic\}$ is fixed by a branching constraint of $z$, then there is a point $x \in P \cap \{x : C_z\}$ with $\sum_{v \in V} x_v \geq m + \frac{1}{2}$ (for example, the point setting $x_{ia} = x_{ib} = x_{ic} = 1/2$ and $x_{ja}, x_{jb}, x_{jc}$ to some integer, feasible value for all $j \not = i$). Let $w \in \T$ be a node of minimum height that includes $\{x : x_n = 1\}$ as a branching constraint (i.e. $w = \arg \min \{h(w) : \{x : x_n = 1\} \in C_w\}$). It is clear that its sibling, call it $w'$, is a node of minimum height that includes the constraint $\{x : x_n = 0\}$. Now we will consider two cases: in the first, suppose $x_n$ is fixed in a branching constraint in the first $m/2$ levels of the tree, i.e. $h(w) \leq \frac{m}{2}$. Therefore, at most $m/2$ variables in $V$ have been fixed in $C_w$. Since $\T \in \mathbb{T}^\ell_{2^{m/2}}$, the subtree rooted at $w$ must have $\leq 2^{m/2} - 1$ leaves, and so there is a leaf $w''$ of the subtree with less than $m/2$ branching constraints on the path from $w$ to $w''$ (i.e. $|C_{w''} \setminus C_{w}| \leq \frac{m}{2}-1$), which implies that at most $m-1$ variables in $V$ have been branched on and therefore there is some $i \in [m]$ such that no variable in $\{\{ia,ib\}, \{ib,ic\}, \{ia,ic\}\}$ is fixed in $C_{w''}$. Therefore, there is a point $x \in P \cap C_{w''}$ such that $\sum_{v \in V} x_v \geq m + \frac{1}{2}$ and $x_n = 1$. Now, consider the cases where $h(w') \geq \frac{m}{2} + 1$, or $x_n$ is never branched on in $\T$. Since $\T \in \mathbb{T}^\ell_{2^{m/2}}$, there must be a leaf $z'$ of $\T$ of height at most $\frac{m}{2}$; notice $x_n$ is not fixed by any constraint in $C_{z'}$. Clearly, there is a point $x \in P \cap C_{z'}$ such that $\sum_{v \in V} x_v \geq m + \frac{1}{2}$ and $x_n = 1$. Therefore, for any $\T \in \mathbb{T}^\ell_{2^{m/2}}$, there is some $x^{\T} \in \T(P)$ with $x_n = 1$. Finally, to show that there is some $\bar{x} \in \bigcap_{\T \in \mathbb{T}^\ell_{2^{m/2}}} \T(P)$, we apply Lemma \ref{lem:height_lemma} with $a = e_n, b = 0, b' = 1$ and $s^1, ..., s^n$ being the points $0, e^1, ..., e^{n-1}$, for example (there are several choices of $n$ affinely independent points in $\{0,1\}^{n-1}$), $R = \bigcup_{\T \in \mathbb{T}^\ell_{2^{m/2}}} x^{\T}$ where $x^{\T}$ is defined as above, and $r = x^{\T}$. The Lemma guarantees that such an $\bar{x}$ exists and $\bar{x}_n > 0$, and therefore $\bar{x} \in P \setminus P_I$. 
\end{proof}

\section{Trees with Bounded Height}\label{section6}

In this section, we introduce, and analyze, an operator that captures the strength of branch-and-bound trees with bounded height. 

\begin{definition}
We denote $\mathbb{T}^h_k = \{\T : h(\T) \leq k\}$ as the set of BB trees with height at most $k$. 
\end{definition}

\begin{definition}
We define the \textit{height $k$ branch-and-bound operator} as follows
$$T^k(P) = \bigcap_{\T \in \mathbb{T}^h_k} \T(P)$$
\end{definition}

The main result of this section is that the operator $T^k(\cdot)$ can be ``squeezed'' between two natural lift-and-project operators: the Balas-Ceria-Cornuéjols sequential convexification and the canonical lift-and-project. This is formalized by the following theorem.

\begin{theorem}\label{thm:bounded_height_main}
Let $P \subseteq [0,1]^n$ be a polytope. Then, 
$$L^k(P) \subsetneq T^k(P) \subsetneq B^k(P)$$
for all $k \in [n]$. Moreover, for any $n$ such that $n-1$ is divisible by $6$, there is a polytope $P \subseteq [0,1]^n$ such that $P_I = L^2(P) \subsetneq T^{(n-1)/6}(P)$.
\end{theorem}

\begin{remark}
The height-$k$ operator $T^k(\cdot)$ and the skew $k$-tree operator $\T^{\pi,k}(\cdot)$ are incomparable in general. The results of Section 3 along with Theorem \ref{thm:bounded_height_main} show that there exist $P$ such that $\T^{\pi,k}(P) \subsetneq T^k(P)$. It is easy to construct a $P$ such that $T^k(P) \subsetneq \T^{\pi,k}(P)$: suppose $P = \{x : x_1,...,x_{k} = 1, x_{k+1},...,x_n = \frac{1}{2}\}$, then note that $\T^{\pi,k}(P) = P$ while $T^k(P) = \emptyset$.
\end{remark}

We prove the above theorem in three separate parts. 

\begin{lemma}\label{lem:bb_vs_bcc}
$T^k(P) \subsetneq B^k(P)$ for all $k \in [n]$. 
\end{lemma}
\begin{proof}
We will start by showing that $T^k(P) \subseteq B^k(P)$ for any $k \in [n]$. In particular, we will show that for any subset $\{i_1,...,i_k\} \subset [n]$, there is a tree $\T \in \mathbb{T}^h_k$ such that $\T(P) \subseteq P_{i_1,...,i_k}$. Specifically, consider a complete tree $\T$ of height $k$ that branches only on variables $\{i_1,...,i_k\}$ (i.e. each of the $2^k$ leaves of $\T$ fixes the variables $\{i_1,...,i_k\}$ to one of the $2^k$ possible settings). Let $z$ be an extreme point of $\T(P)$, so $z$ is in the atom of a leaf $v \in \leaves(\T)$. Therefore, $z \in \{x : C_v\} \subset \{x : x_j \in \{0,1\} \ \forall \ j \in \{i_1,...,i_k\}\} \subset P_{i_1,...,i_k}$, where the last inclusion is by definition of the latter set. 

Now, we will give an example in $\R^3$ where the inclusion $T^k(P) \subset B^k(P)$ is strict. Let 
$$P = \conv(\{(0,0,0), (1,0,0), (0,1,1), (1,1,1), (0,0,\frac{1}{2}), (\frac{1}{2},0,1), (1,\frac{1}{2},1)\})$$ 
and observe that $P_I = \{x : ax = b\} \cap [0,1]^n$ for $a = (0,-1,1)$ and $b = 0$. Let $\T \in \mathbb{T}^h_2$ be a tree with leaves $v^1,v^2,v^3,v^4$ where $C_{v^1} = \{x_1 = 0, x_3 = 0\}, C_{v^2} = \{x_1 = 0, x_3 = 1\}, C_{v^3} = \{x_1 = 1, x_2 = 0\}, C_{v^4} = \{x_1 = 1, x_2 = 1\}$ and observe that $\T(P) = P_I$, and therefore $T^2(P) = P_I$. We will now show that $B^2(P) \not = P_I$. Note that $(0,0,\frac{1}{2}) \in P_{1,2}$, $(\frac{1}{2}, 0, 1) \in P_{2,3}$, and $(1,\frac{1}{2},1) \in P_{1,3}$, and of course $P_I \subset P_{i,j}$ for all $i,j \in [3]$. Therefore, we see that there is a point $\bar{x} \in B^2(P) = \bigcap_{i,j \in [3]} P_{i,j}$ with $a\bar{x} > 0$ by applying Lemma \ref{lem:height_lemma} with $a,b$ as defined above, $b' = \frac{1}{2}$ and $R = \{(0,0,\frac{1}{2}), (\frac{1}{2},0,1), (1,\frac{1}{2},1)\}$, and $s^1,...,s^3$ being any three affinely independent points of $P_I$. 
\end{proof}

\begin{lemma}\label{lem:bb_vs_lp}
$L^k(P) \subsetneq T^k(P)$ for all $k \in [n]$.
\end{lemma}
\begin{proof}
We will prove the lemma by induction on $k$. First, observe that the definitions imply $L(P) = T(P)$ (since $x \in \conv(\{x \in P : x_j \in \{0,1\}\})$ if and only if $x \in \T(P)$ for some $\T$ of height $1$). Now, assuming that $L^k(P) \subseteq T^k(P)$, we will show that $L^{k+1}(P) \subseteq T^{k+1}(P)$.

Note $L^{k+1}(P) = \bigcap_{j \in [n]} \conv(L^k(\{x \in P : x_j = 0\}) \cup L^k(\{x \in P : x_j = 1\}))$ by definition of $L^{k+1}(\cdot)$ and Lemma \ref{lem:nest_l}. Therefore, by the induction hypothesis, $L^{k+1}(P) \subseteq \bigcap_{j \in [n]} \conv(T^k(\{x \in P : x_j = 0\}) \cup T^k(\{x \in P : x_j = 1\}))$. We will proceed by showing
$$\bigcap_{j \in [n]} \conv\left(T^k(\{x \in P : x_j = 0\}) \cup T^k(\{x \in P : x_j = 1\})\right) \subseteq T^{k+1}.$$
Let $x \in \bigcap_{j \in [n]} \conv\left(T^k(\{x \in P : x_j = 0\}) \cup T^k(\{x \in P : x_j = 1\})\right)$, we will show that $x \in \T(P)$ for any $\T \in \mathbb{T}^h_{k+1}$. Let $j \in [n]$ be the variable branched on at the root of $\T$, and let $\T^0$ be the subtree rooted at the child of the root corresponding to branch $x_j = 0$ (and let $\T^1$ be defined similarly). Observe that $\T^0,\T^1 \in \mathbb{T}^h_k$. We know that $x \in \conv(x^{j,0}, x^{j,1})$, for some $x^{j,a} \in T^k(\{x \in P : x_j = a\})$ for $a \in \{0,1\}$. By definition of $T^k(\cdot)$ and the fact that $\T^0 \in \mathbb{T}^h_k$, we know that $x^{j,0} \in \T^0(\{x \in P : x_j = 0\})$ (and similarly $x^{j,1} \in \T^1(\{x \in P : x_j = 1\})$). Therefore, $x^{j,0}, x^{j,1} \in \T(P)$, which clearly implies $x \in \T(P)$ as desired.
\end{proof}


The following is a direct corollary of Theorem \ref{thm:bb_vs_lp}.

\begin{cor}\label{cor:bb_vs_lp_height}
For any $n$ such that $n-1$ is divisible by $6$, there is a polytope $P \subseteq [0,1]^n$ such that $P_I = L^2(P) \subsetneq T^{(n-1)/6}(P)$.
\end{cor}

\begin{proof}[Proof of Theorem \ref{thm:bounded_height_main}]
The theorem is directly implied by the combination of Lemma \ref{lem:bb_vs_bcc}, Lemma \ref{lem:bb_vs_lp}, and Corollary \ref{cor:bb_vs_lp_height}.
\end{proof}

\begin{remark}
There is a simple construction, of similar spirit to that of Theorem \ref{thm:bb_vs_lp}, that shows that for any $n \geq 3$, there is a polytope $P \subseteq [0,1]^n$ such that $P_I = L^2(P) \subsetneq T^{n-1}(P)$.
\end{remark}

While at first glance, the Balas-Ceria-Cornuéjols operator and the canonical lift-and-project operator do not seem to differ significantly, the following direct corollary of Theorem \ref{thm:bounded_height_main} shows that applying the intersection ``as you go'' as opposed to applying it at the end of the disjunctive procedure can in fact have quite a significant impact. 
\begin{cor}
For any $n$ such that $n-1$ is divisible by $6$, there is a polytope $P \subseteq [0,1]^n$ such that $P_I = L^2(P) \subsetneq B^{(n-1)/6}(P)$.
\end{cor}

\begin{remark}
Theorem \ref{thm:bounded_height_main} gives a necessary condition for branch-and-bound to be advantagous compared to lift-and-project. In particular, if the best compact formulation constructed by branch-and-bound $Q_{BB}$ comes from a ``short, balanced'' tree (i.e. a tree in $\mathbb{T}^h_k$), Sherali-Adams is also able to construct a compact formulation $Q_{SA}$ such that $\max\{ cx : (x,y) \in Q_{SA}\} \leq  \max\{ cx : (x,y) \in Q_{BB}\}$. 
\end{remark}


{\bf Acknowledgements:} This research was supported in part by ONR grant N00014-22-1-2528.

\bibliography{bib}
\bibliographystyle{alpha}

\end{document}